\documentclass[reqno]{amsart}

\usepackage{amsmath}
\usepackage[nobysame]{amsrefs}
\usepackage{amssymb, color}
\usepackage[margin=1in]{geometry}

\usepackage{mathrsfs}
\usepackage{graphicx}
\usepackage{subfigure}
\usepackage{float}
\usepackage{epsf}
\usepackage{hyperref}
\usepackage{titletoc}
\usepackage{ulem}

\graphicspath{{../Figures/}}

\numberwithin{equation}{section}

\newtheorem{lem}{Lemma}[section]
\newtheorem{thm}{Theorem}[section]
\newtheorem{prop}[thm]{Proposition}

\theoremstyle{remark}
\newtheorem{rmk}{Remark}[section]

\renewcommand{\tilde}{\widetilde}
\renewcommand{\hat}{\widehat}

\newcommand{\nn}{\nonumber}
\newcommand{\R}{{\mathbb R}}

\newcommand{\Grad}{\nabla_{\!x}}
\newcommand{\del}{\partial}
\newcommand{\dx}{ \, {\rm d} x}

\newcommand{\dv}{ \, {\rm d} v}
\newcommand{\dvbar}{ \, {\rm d} \bar v}
\newcommand{\ds}{\, {\rm d} s}
\newcommand{\dmu}{\, {\rm d} \mu}
\newcommand{\dtau}{\, {\rm d} \tau}
\newcommand{\dxi}{\, {\rm d} \xi}
\newcommand{\Denote}{\stackrel{\triangle}{=}}
\newcommand{\Ni}{\noindent}

\newcommand{\Ss}{{\mathbb{S}}}

\newcommand{\ii}{{\mathcal I}}

\newcommand{\CalA}{{\mathcal{A}}}

\newcommand{\CalD}{{\mathcal{D}}}

\newcommand{\CalK}{{\mathcal{K}}}
\newcommand{\CalL}{{\mathcal{L}}}

\newcommand{\CalT}{{\mathcal{T}}}

\newcommand{\EpsJ}{J^\Eps}

\newcommand{\VecL}{{\CalL}}

\newcommand{\Eps}{\epsilon}

\newcommand{\NullL} {{\rm Null} \, \VecL}

\newcommand{\abs}[1]{\left\lvert#1\right\rvert}
\newcommand{\norm}[1]{\left\lVert#1\right\rVert}
\newcommand{\vint}[1]{\left\langle#1\right\rangle}
\newcommand{\vpran}[1]{\left(#1\right)}

\DeclareMathOperator{\Span}{span}

\begin{document}

\title{Fractional Diffusion Limits of Non-Classical Transport Equations}

\author{Martin Frank}
\author{Weiran Sun}


\begin{abstract}
We establish asymptotic diffusion limits of the non-classical transport equation derived in \cite{Lar07}. 
By introducing appropriate scaling parameters, the limits will be either regular or fractional diffusion equations depending on the tail behaviour of the path-length distribution. Our analysis uses the Fourier transform combined with a moment method. We conclude with remarks on the diffusion limit of the periodic Lorentz gas equation. 
\end{abstract}

\maketitle

\section{Introduction}
Anomalous diffusion, a diffusion process described by a fractional diffusion equation, has gained a lot of interest recently. Examples include L\'evy glasses \cite{VynBurRibWie12}, plasma physics \cite{CarLynZas01}, spreading of diseases \cite{SchHanDel11}, chemical reactions \cite{AlbMar88}, elementary particle physics \cite{SagBroAmoDav12}, and flight patterns of birds \cite{VisAfaBulMurPriSta96}. Many more examples are contained in the aptly-titled review \cite{MetKla00}.

In most works, the argument for coming up with an equation involving the fractional Laplacian $(-\Delta)^{\alpha/2}$ is a scaling argument: The Green's function associated to the fractional Laplacian has a tail that decays algebraically like $x^{-\alpha}$. If the data has a similar scaling behavior, then the underlying system is modeled by a fractional diffusion equation. Fractional diffusion can be rigorously derived from Continuous Time Random Walks (CTRWs) in the limit of many interactions by some Generalized Central Limit Theorem \cite{MetKla00}. However, there is often no microscopic picture that yields this random walk.

It is therefore a mathematical challenge to provide a microscopic picture, and rigorously derive macroscopic equations. One possible  strategy to address this challenge comes from kinetic theory, where the passage from particle transport in a random medium, via a kinetic description, to macroscopic equations is well understood \cite{Cercignani}. Historically, this has led to many insights, not the least of which is the understanding of the fluid dynamic equations as limits of the Boltzmann equation. 

To our knowledge, the first rigorous mathematical work to prove convergence of solutions of classical transport equations to solutions of fractional diffusion equations is \cite{MelMisMou11}. The authors use a Fourier technique which formally already has been known in the fractional calculus literature (cf.\ \cite{ScaGorMaiRab03}). See also the related works~\cite{AbdMelPue11, NMP2011} where fractional diffusion equations can arise from classical transport equations. 

The starting point for our work is the non-classical transport equation proposed by Larsen \cite{Lar07} (see \eqref{eq2} for the explicit equation). The original motivation for this equation was from measurements of photon path-length in atmospheric clouds, which could not be explained by classical radiative transfer, cf.\ \cite{Pfe99} or sections 5.1 and 8.3 in the review \cite{DavMar10}. Classically, the amount of radiation, when it passes through a medium, is attenuated exponentially. This is the well-known law of Beer-Lambert. Recent measurements, however, have revealed that radiation through an atmospheric cloud is attenuated less, namely merely algebraically \cite{Pfe99}. This has led Larsen to formulate a Boltzmann equation on an extended phase space \cite{Lar07}, which he named non-classical transport equation. The equation is able to model particle transport with given path-length distributions $p(s)$, $s$ being the path-length, and $p$ its probability density function.
Non-classical transport theory has since been extended \cite{LarVas11} and has found applications for neutron transport in pebble bed reactors \cite{VasLar09}, and even computer graphics \cite{Eon14}. 

In his original paper \cite{Lar07}, Larsen has considered the formal diffusion limit of the non-classical transport equation. This has been made rigorous in \cite{FraGou10}. However, the classical analysis cannot capture the case when the second moment, i.e.\ the variance, of the path-length distribution does not exist. The purpose of this paper is to extend the analysis to cover this case and make the limit process rigorous. It will turn out that in the case of an infinite variance of the path-length distribution, the limiting equation is a fractional diffusion equation. This paper therefore provides a connection between non-classical transport and anomalous diffusion. The result is stated mathematically in Section~\ref{sec:formulation}. In Section~\ref{sec:well-posedness} we give a short proof of the well-posedness of the transport equation, which lays down the basic functional setting in this paper. The main part is in Section~\ref{sec:limit} where we establish various limits of the transport equation. 

The connection to a microscopic picture becomes somewhat complete because non-classical transport theory can be connected to random walks in a specific physical medium. Recent results by Golse et al.\ (cf.\ \cite{Gol12} for a review), and by Marklof \& Str\"ombergsson \cite{MarStr11} show that an equation similar to the non-classical transport equation can be derived from particle transport in a regular lattice (the so-called periodic Lorentz gas equation). In 2D, an explicit path-length distribution can be computed. Marklof \& T\'oth \cite{MarTot14} proved a superdiffusive central limit theorem for the particle billiards and showed that the periodic Lorentz gas is superdiffusive (but only logarithmically). We are able to reproduce a result in the same spirit for the simpler case of non-classical transport, using techniques from kinetic theory. We comment on this in Section~\ref{sec:lorentz}.

\section{Main Result} \label{sec:formulation}
The non-classical transport equation with a scaling parameter $\Eps$ as considered in \cite{Lar07} has the form 
\begin{equation}
\begin{gathered} \label{eq2}
\frac{1}{\Eps} \partial_s \psi_\Eps(x,v,s) + v\cdot\nabla_x \psi_\Eps(x,v,s) + \frac{\Sigma_t(s)}{\Eps}\psi_\Eps(x,v,s) 
\\
 = \delta(s)\int_{S^{n-1}}\int_0^\infty \vpran{\sigma(v \cdot v') -\theta(\Eps)(1-c)}\frac{\Sigma_t(s')}{\Eps} \psi_\Eps(x,v',s') \ds' \dv' + \delta(s) \frac{\theta(\Eps)}{\Eps}\frac{Q}{4\pi}.
\end{gathered}
\end{equation}
The unknown function $\psi_\Eps$ is the angular flux of particles at position $x\in\R^n$, moving into direction $v \in S^{n-1}$ (unit vector). The particles interact with a background medium. The interaction of the particles is described by the collision cross section $\Sigma_t$. What makes the equation non-classical is that $\Sigma_t=\Sigma_t(s)$ depends on the distance $s$ from the last collision. 
The angular scattering kernel $\sigma(v \cdot v')$ is independent of $s$. Moreover, the measure $\dv$ is scaled to be the unit measure on $\Ss^{n-1}$ and $\sigma$ satisfies that
\begin{align} \label{assump:sigma}
   \int_{\Ss^{n-1}} \sigma(v \cdot v') \dv = 1 \,.
\end{align}
The equation is completed by the particle source $Q$, and the scattering ratio $c$ (when a particle interacts with the background, the probability that it is absorbed is $1-c$, the probability that it scatters is $c$). We will assume throughout the paper that $c<1$, i.e.\ there is a small amount of absoprtion everywhere. The Dirac delta $\delta(s)$ on the right-hand side models that particles which scatter have their distance-to-previous-collision reset to zero. In the case of constant $\Sigma_t$, this equation reduces to the classical transport equation~\cite{Lar07}.

The parameter $\Eps$ being small means that we have many collisions (small Knudsen number). Extending the scaling in \cite{Lar07} where $\theta(\Eps) = \Eps^2$, we have introduced a general function $\theta(\Eps)$ to scale the obsorption term and the source. We assume that $\theta(\Eps)$ is monotonically increasing with $\Eps$ and $\theta(\Eps)\to 0$ as $\Eps\to 0$. In most cases, we will later use $\theta(\Eps)=\Eps^\alpha$ with $1<\alpha < 2$. Some comments on this particular choice of the scaling are in order: First, as in \cite{Lar07}, we have fixed the scale of $\Sigma_t$ to be $1/\Eps$, which means the scattering mean free path is small. Let $p(s)$ be the path-length distribution defined by
\begin{equation}
p(s) = \Sigma_t(s) \exp(-\int_0^s \Sigma_t(s') \ds') \,.
\end{equation}
The scales of $s$ and $\Sigma_t(s)$ are related in the way such that $p$ integrates to one for any $\Eps$. Thus $s$ has to be rescaled by $\Eps$.
Second, if we rearrange the equation as
\begin{gather*}
\frac{1}{\Eps} \partial_s \psi_\Eps(x,v,s) 
+ v\cdot\nabla_x \psi_\Eps(x,v,s) 
+ \frac{\Sigma_t(s)}{\Eps}\psi_\Eps(x,v,s) 
- \delta(s)\int\int_0^\infty 
                  \sigma(v \cdot v')\frac{\Sigma_t(s')}{\Eps} \psi_\Eps(x,\Omega',s') \ds' \dv' 
 \\
 = \delta(s)\theta(\Eps) \left(  \frac{1}{\Eps}\frac{Q}{4\pi} 
 - (1-c)\int \int_0^\infty 
     \frac{\Sigma_t(s')}{\Eps} \psi_\Eps(x,v',s')\ds' \dv' \right).
\end{gather*}
then it becomes clear that the factor $\theta(\Eps)$ controls the relative weakness of emission/absorption compared to scattering. 
Therefore there can only remain one relative scaling factor, which we have called $\theta(\Eps)$. 

Our main purpose of this paper is to prove the following convergence result as $\Eps \to 0$:
\begin{thm}
Suppose the scattering constant $c$ and the cross section $\sigma$ satisfy the assumptions
\begin{align} \label{assump:sigma}
   0 < c < 1 \,,
\qquad
   \int_{\Ss^{n-1}} \sigma(v \cdot v') \dv' = 1 \,,
\qquad
   \sigma(v\cdot v') \geq \sigma_0 > 0
\end{align}
for some constant $\sigma_0$. Suppose the path-length distribution function $p$ satisfies 
\begin{equation}\label{assump:p}
   \int_0^\infty p(s) \ds = 1, \qquad \int_0^\infty s p(s) \ds < \infty \,,\qquad p(s) = \frac{d_0}{s^{\alpha + 1}} \quad   \text{for $s > 1$,}
\tag{B}
\end{equation}
where $d_0 > 0$ is a constant. 
Let  $\Psi_\Eps(s,x,v) = \psi_\Eps(s,x,v) e^{\int_0^s \Sigma_t(\tau) \dtau}$, where $\psi_\Eps \in L^\infty(0, \infty; L^2(\R^n \times \Ss^{n-1}))$ is the solution to~\eqref{eq2}.
Then there exists $\Psi_0 \in L^2(\R^n)$ which only depends on $x$ such that
\begin{align*}
    \Psi_\Eps \to \Psi_0
\qquad in \,\,
   w^\ast-L^\infty(0, \infty; L^2(\R^n \times \Ss^{n-1})).
\end{align*}
Furthermore, there exists $q \in L^2(\R^n \times \Ss^{n-1})$ such that with the following choices of $\theta(\Eps)$, the limit $\Psi_0$ satisfies the (fractional) diffusion equation
\begin{itemize}
\item[(a)]  $ D_1 (-\Delta) \Psi_0 + (1 - c) \Psi_0 = \int_{\Ss^{n-1}} Q(x, v) \dv$ if $\alpha > 2$ and $\theta(\Eps)=\Eps^2$;
\item[(b)]  $ D_2 (-\Delta)^{\alpha/2} \Psi_0 + (1 - c) \Psi_0 = \int_{\Ss^{n-1}} Q(x, v) \dv$ if $1 < \alpha < 2$ and $\theta(\Eps)=\Eps^\alpha$;
\item[(c)]  $ D_3 (-\Delta) \Psi_0 + (1 - c) \Psi_0 = \int_{\Ss^{n-1}} Q(x, v) \dv$ if $\alpha = 2$ and $\theta(\Eps)=-\Eps^2\ln\Eps$,
\end{itemize}
where the positive coefficients $D_1, D_2, D_3$ can be explicitly computed from $p$ and $\sigma$.
\end{thm}


\section{Well-posedness} \label{sec:well-posedness}
In this section we establish the well-posedness of the transport equation in the spaces $L^\infty(0, \infty; L^q(\R^n \times \Ss^{n-1}))$ for any $1 \leq q \leq \infty$. This can be done either by applying the iteration method used in \cite{FraGou10} or by using a fixed-point argument. Here we employ the latter method.  

Let 
$$
\Psi_\Eps(s, x, v) = \psi_\Eps(s, x, v) e^{\int_0^s \Sigma_t(\tau) \dtau}. 
$$
Eq.\ \eqref{eq2} can be re-written as \cite{Lar07}
\begin{equation}
\begin{gathered} \label{eq:transport-1}
   \frac{1}{\Eps}\del_s \Psi_\Eps + v \cdot \Grad \Psi_\Eps = 0 \,,
\\
   \Psi_\Eps(0, x, v) 
  =  \int_0^\infty\int_{\Ss^{n-1}} 
         \vpran{ \sigma(v \cdot v') - \theta(\Eps) (1-c)} p(\tau) \Psi_\Eps(\tau, x, v') \dv'\dtau
       + \theta(\Eps) Q(x, v) \,.
\end{gathered}
\end{equation}
We can further re-formulate equation~\eqref{eq:transport-1} using characteristics. This gives
\begin{align} \label{def:soln}
   \Psi_\Eps(s, x, v)
   = \int_0^\infty\int_{\Ss^{n-1}} 
         \vpran{\sigma(v \cdot v') - \theta(\Eps) (1-c)} p(\tau)
         \Psi_\Eps(\tau, x - \Eps vs, v') \dv'\dtau
       + \theta(\Eps) Q(x- \Eps vs, v) \,.
\end{align}
It is this last formulation that we will use to carry out our analysis in this paper. 

The well-posedness result is 
\begin{thm} \label{thm:well-posedness}
Suppose the scattering coefficient $c$ and the cross section $\sigma \geq 0$ satisfy the conditions
\begin{align} \label{assump:sigma-1}
   0 < c < 1 \,,
\qquad
   \int_{\Ss^{n-1}} \sigma(v \cdot v') \dv = 1 \,,
\qquad
   \sigma(v\cdot v') \geq \sigma_0 > 0 
\end{align}
for some constant $\sigma_0> 0$. Suppose the path-length distribution function $p$ and the source term satisfy   
\begin{align} \label{cond:p-Q}
   \int_{0}^\infty p(s) \ds = 1 \,,
\qquad
   Q \in L^q(\R^n \times \Ss^{n-1})
\quad \,
   \text{for any $1 \leq q \leq \infty$}.
\end{align}
Then for each fixed $\Eps > 0$ small enough such that $\sigma - \theta(\Eps) (1-c) \geq 0$, equation~\eqref{eq:transport-1} has a unique solution $\Psi_\Eps \in L^\infty((0, \infty); L^q(\R^n \times \Ss^{n-1}))$ in the sense of \eqref{def:soln}.  Moreover, $\Psi_\Eps$ satisfies the uniform-in-$\Eps$ bound 
\begin{align} \label{bound:Psi-Eps}
   \norm{\Psi_\Eps}_{L^\infty(0, \infty; L^q(\R^n \times \Ss^{n-1}))}
\leq  
   \frac{1}{1-c} \norm{Q}_{L^q(\R^n \times \Ss^{n-1})} \,,
\qquad  
   1 \leq q \leq \infty \,.
\end{align}
Furthermore, if $Q \geq 0$, then $\Psi_\Eps \geq 0$.
\end{thm}

Before proving Theorem~\ref{thm:well-posedness}, we state a simple lemma that will be used frequently in this paper:
\begin{lem} \label{lem:CS}
Suppose $c, \sigma, p$ satisfy \eqref{assump:sigma-1}-\eqref{cond:p-Q} and $u \in L^\infty\vpran{0, \infty; L^q(\R^n \times \Ss^{n-1})}$ for any $1 \leq q \leq \infty$.~Then 
\begin{align} \label{bound:CS}
& \quad \,
   \norm{\int_0^\infty\int_{\Ss^{n-1}} 
         \vpran{\sigma(v \cdot v')  - \theta(\Eps) (1-c)}p(\tau)
         u(\tau, x - \Eps vs, v') \dv'\dtau}_{L^\infty(0, \infty; L^q(\R^n \times \Ss^{n-1}))} \nn
\\[2pt]
&\leq
    \vpran{1 - \theta(\Eps) (1-c)} \norm{u}_{L^\infty(0, \infty; L^q(\R^n \times \Ss^{n-1}))}
\end{align}
for any $\Eps > 0$. 
In the limit case $c=1$ we have
\begin{align} \label{bound:CS-1}
   \norm{\int_0^\infty\int_{\Ss^{n-1}} 
         \sigma(v \cdot v') p(\tau)
         u(\tau, x - vs, v') \dv'\dtau}_{L^\infty(0, \infty; L^q(\R^n \times \Ss^{n-1}))}
&\leq
    \norm{u}_{L^\infty(0, \infty; L^q(\R^n \times \Ss^{n-1}))} \,.
\end{align}
\end{lem}
\begin{proof}
This result follows directly from the Minkowski and Cauchy-Schwarz inequalities. Denote 
\begin{align*}
     \tilde \sigma = \sigma - \theta(\Eps) (1-c) \geq 0 \,.
\end{align*}
The case $p=\infty$ follows directly from the normalization conditions for $\sigma$ and $p$. If $1 \leq q < \infty$, then integrating in $x$ gives
\begin{align*}
   \norm{\int_0^\infty\int_{\Ss^{n-1}} 
         \tilde \sigma(v \cdot v') p(\tau)
         u(\tau, x-\Eps vs, v') \dv'\dtau}_{L^q(\R^n)}
&\leq
  \int_0^\infty\int_{\Ss^{n-1}} 
         \tilde \sigma(v \cdot v') p(\tau)
         \norm{u(\tau, \cdot, v')}_{L^q(\R^n)} \dv'\dtau
\\
& \hspace{-6cm}
\leq 
   \abs{\int_0^\infty\int_{\Ss^{n-1}} 
         \tilde \sigma(v \cdot v') p(\tau) \dv'\dtau}^{1/{q^\ast}}
   \abs{\int_0^\infty\int_{\Ss^{n-1}} 
         \tilde \sigma(v \cdot v') p(\tau)
         \norm{u(\tau, \cdot, v')}_{L^q(\R^n)}^q \dv'\dtau}^{1/q}
\\
& \hspace{-6cm}
=  \vpran{1 - \theta(\Eps) (1-c)}^{\frac{1}{q^\ast}} \abs{\int_0^\infty\int_{\Ss^{n-1}} 
         \tilde \sigma(v \cdot v') p(\tau)
         \norm{u(\tau, \cdot, v')}_{L^q(\R^n)}^q \dv'\dtau}^{1/q}
\end{align*}
where $\frac{1}{q} + \frac{1}{q^\ast} = 1$. Then by integrating in $v$  we get
\begin{align*} 
   \norm{\int_0^\infty \!\!\!\! \int_{\Ss^{n-1}}  \!
         \sigma(v \cdot v') p(\tau)
         u(\tau, x - \Eps vs, v') \dv'\dtau}_{L^q(\R^n \times \Ss^{n-1})}
&\leq
    \vpran{1 - \theta(\Eps) (1-c)} \norm{u}_{L^\infty(0, \infty; L^q(\R^n \times \Ss^{n-1}))},
\end{align*}
which gives the desired bound.
\end{proof}

Now we proceed to prove Theorem~\ref{thm:well-posedness}.
\begin{proof}[Proof of Theorem~\ref{thm:well-posedness}]
For each fixed $\Eps > 0$ small enough such that $\sigma - \theta(\Eps) (1-c) \geq 0$, define the operator $\CalT$ on $L^\infty(0, \infty; L^q(\R^n \times \Ss^{n-1}))$ as
\begin{align*}
    \CalT u 
   = \int_0^\infty\int_{\Ss^{n-1}} 
         \vpran{\sigma(v \cdot v') - \theta(\Eps) (1-c)} p(\tau)
         u(\tau, x - \Eps vs, v') \dv'\dtau
       + \theta(\Eps) Q(x- \Eps vs, v) \,.
\end{align*}
We will show that $\CalT$ maps $L^\infty(0, \infty; L^q(\R^n \times \Ss^{n-1}))$ into itself and $\CalT$ is a contraction mapping. 

First, the source term in $\CalT$ satisfies
\begin{align} \label{bound:T-0}
    \int_{\Ss^{n-1}} \int_{\R^n}
       \abs{Q(x- \Eps vs, v)}^q \dx \dv
   = \int_{\Ss^{n-1}} \int_{\R^n} \abs{Q(x, v)}^q \dx \dv \,.
\end{align}
Applying Lemma~\ref{lem:CS} to the integral term in $\CalT$, we have
\begin{align} \label{bound:T-2}
    \norm{\CalT u}_{L^\infty(0, \infty; L^q(\R^n \times \Ss^{n-1}))}
\leq
   (1 - \theta(\Eps) (1-c)) \norm{u}_{L^\infty(0, \infty; L^q(\R^n \times \Ss^{n-1}))}
   + \theta(\Eps) \norm{Q}_{L^q(\R^n \times \Ss^{n-1})}
< \infty \,.
\end{align}
Hence $\CalT$ maps $L^\infty(0, \infty; L^q(\R^n \times \Ss^{n-1}))$ into itself. 

Next, we show that $\CalT$ is a contraction mapping. To this end, let $u, v \in L^\infty(0, \infty; L^q(\R^n \times \Ss^{n-1}))$. Applying \eqref{bound:CS} to $u-v$, we get
\begin{align*}
   \norm{\CalT u - \CalT v}_{L^\infty(0, \infty; L^q(\R^n \times \Ss^{n-1}))}
\leq 
  (1- \theta(\Eps) (1-c)) \norm{u-v}_{L^\infty(0, \infty; L^q(\R^n \times \Ss^{n-1}))} \,.
\end{align*}
which shows that $\CalT$ is a contraction mapping since the coefficient satisfies $1- \theta(\Eps) (1-c) < 1$ for each fixed $\Eps > 0$. 
We can now apply the fixed-point theorem to conclude that equation~\eqref{def:soln} has a unique solution $\Psi_\Eps \in L^\infty(0, \infty; L^q(\R^n \times \Ss^{n-1}))$ for any $1 \leq q \leq \infty$.

The uniform bound of $\Psi_\Eps$ in \eqref{bound:Psi-Eps} follows directly from \eqref{bound:T-2} with $\CalT u$  and $u$ in the inequality being replaced by $\Psi_\Eps$. The positivity of $\Psi_\Eps$ can be obtained by noting that the mapping $\CalT$ preserves positivity if $Q \geq 0$. Hence the unique solution obtained through iterations in the fixed-point argument must be non-negative. 
\end{proof}

\section{Passing to the Limit} \label{sec:limit}
In this section we show the limit of \eqref{def:soln} as $\Eps \to 0$. Roughly speaking, the main result is to recover a regular or fractional diffusion equation in the limit for the quantity $\int_{(0, \infty) \times \Ss^{n-1}} \psi_\Eps(s, x, v) \dv\ds$. 
Throughout this section, 
we assume that $Q \in L^1 \cap L^2 (\R^n \times \Ss^{n-1})$. Then by Theorem~\ref{thm:well-posedness}, the solution $\Psi_\Eps$ is uniformly bounded in $L^\infty(0, \infty; L^1 \cap L^2(\R^n \times \Ss^{n-1}))$ with bounds satisfying \eqref{bound:Psi-Eps} for $q=1,2$. 

First, we show the convergence of $\Psi_\Eps$ as $\Eps \to 0$. 
\begin{thm} \label{thm:Psi}
Let $Q \in L^1 \cap L^2(\R^n \times \Ss^{n-1})$. Then there exists a subsequence $\Psi_{\Eps_k}$ and $\Psi_0=\Psi_0(x)$ such that 
\begin{align} \label{converg:Psi-Eps-k}
    \Psi_{\Eps_k} \to \Psi_0 
\qquad
    \text{in $w^\ast-L^\infty(0, \infty; L^2(\R^n \times \Ss^{n-1}))$.}
\end{align}
\end{thm}
\begin{proof}
The convergence of a subsequence of $\Psi_\Eps$ is guaranteed by the uniform bound of $\Psi_\Eps$ in~\eqref{bound:Psi-Eps}. Therefore there exists $\Psi_0 \in L^\infty(0, \infty; L^2(\R^n \times \Ss^{n-1}))$ such that~\eqref{converg:Psi-Eps-k} holds. What remains to show is that the limiting function $\Psi_0$ is independent of $s, v$. Our main goal is to prove that $\Psi_0$ satisfies
\begin{align} \label{eq:Psi-0}
    \Psi_0(s, x, v) = \int_0^\infty \int_{\Ss^{n-1}} \sigma(v\cdot v') p(\tau) \Psi_0(\tau, x, v') \dv' \dtau \,.
\end{align}
Then the non-negativity of $\Psi_\Eps$ and the unity of the measure $\sigma(v \cdot v') p(\tau) \dv' \dtau$ imply that $\Psi_0 = \Psi_0(x)$. 

In order to show~\eqref{eq:Psi-0}, we recall that $\Psi_\Eps$ satisfies
\begin{align} \label{def:soln-1}
   \Psi_\Eps(s, x + \Eps vs, v)
   = \int_0^\infty\int_{\Ss^{n-1}} 
         \vpran{\sigma(v \cdot v') - \theta(\Eps) (1-c)} p(\tau)
         \Psi_\Eps(\tau, x, v') \dv'\dtau
       + \theta(\Eps) Q(x, v) \,.
\end{align}
We will prove that~\eqref{def:soln-1} converges to~\eqref{eq:Psi-0} as $\Eps \to 0$ in the sense of distributions. First we study the convergence of the right-hand side of~\eqref{def:soln-1}. Note that the right-hand side of~\eqref{def:soln-1} is independent of $s$. Hence any convergence is uniform in $s$.  The terms associated with $\theta(\Eps)$ satisfy 
\begin{align*}
   \theta(\Eps) Q \to 0 \,,
\qquad
   \theta(\Eps) (1-c) \int_0^\infty \int_{\Ss^{n-1}} p(\tau) \Psi_\Eps(\tau, x, v') \dv'\dtau \to 0
\quad 
   \text{in $L^\infty(0, \infty; L^2(\R^n \times \Ss^{n-1}))$}
\end{align*}
by the uniform bound in~\eqref{bound:Psi-Eps} and~\eqref{bound:CS-1}. Next, for any $h_1(x, v) \in L^2(\R^n \times \Ss^{n-1})$, 
\begin{align} \label{converg:h-1}
& \quad \,
   \int_{\R^n} \int_{\Ss^{n-1}} h_1(x, v) \vpran{\int_0^\infty \int_{\Ss^{n-1}} \sigma(v \cdot v') p(\tau) \Psi_\Eps(\tau, x, v') \dv'\dtau} \dv\dx \nn
\\
& = \int_0^\infty \int_{\R^n} \int_{\Ss^{n-1}}     
          \vpran{\int_{\Ss^{n-1}} h_1(x, v) \sigma(v \cdot v') \dv} p(\tau)
          \Psi_\Eps(\tau, x, v') \dv'\dx\dtau 
\\
& \to \int_0^\infty \int_{\R^n} \int_{\Ss^{n-1}}     
          \vpran{\int_{\Ss^{n-1}} h_1(x, v) \sigma(v \cdot v') \dv} p(\tau)
          \Psi_0(\tau, x, v') \dv'\dx\dtau \,. \nn
\end{align}
where the last step follows from~\eqref{converg:Psi-Eps-k} and Lemma~\ref{lem:CS}, since by~\eqref{bound:CS-1} (with $s=0$) we have
\begin{align*}
   \vpran{\int_{\Ss^{n-1}} h_1(x, v) \sigma(v \cdot v') \dv'} p(\tau)  
 \in L^1(0, \infty; L^2(\dx\dv)) \,.
\end{align*}
Therefore as $\Eps \to 0$,
\begin{align} \label{limit:RHS}
     \text{RHS of ~\eqref{def:soln-1} $\longrightarrow$ RHS of ~\eqref{eq:Psi-0} in $\CalD'$.}
\end{align}
Next we show the convergence of the left-hand side of~\eqref{def:soln-1}. To this end, let $h_2 \in ~C_c^\infty((0, \infty) \times \R^n \times \Ss^{n-1})$. 
Then the left-hand side term satisfies
\begin{align*}
& \quad \,
   \int_0^\infty \int_{\R^n} \int_{\Ss^{n-1}}
       h_2(s, x, v) \Psi_\Eps(s, x + \Eps vs, v) \dv\dx\ds
\\
& = \int_0^\infty \int_{\R^n} \int_{\Ss^{n-1}}
       h_2(s, x, v) \Psi_\Eps(s, x, v) 
     + \int_0^\infty \int_{\R^n} \int_{\Ss^{n-1}}
       h_2(s, x, v) \vpran{\Psi_\Eps(s, x + \Eps vs, v) - \Psi_\Eps(s, x, v)} 
\\
& \Denote  I_{1, \Eps} + I_{2, \Eps} \,.
\end{align*}
By \eqref{converg:Psi-Eps-k}, the limit of $I_{1, \Eps}$ is
\begin{align} \label{limit:I-1-Eps}
    I_{1, \Eps} 
\to \int_0^\infty \int_{\R^n} \int_{\Ss^{n-1}}
       h_2(s, x, v) \Psi_0(s, x, v) \dv\dx\ds \,.
\end{align}
Denote the compact support of $h_2$ as $\Omega$. Then the limit of $I_{2, \Eps}$ is
\begin{align} \label{limit:I-2-Eps}
  \abs{ I_{2, \Eps}}
&= \abs{\int_0^\infty \int_{\R^n} \int_{\Ss^{n-1}}
       \vpran{h_2(s, x - \Eps vs, v) - h_2(s, x, v)} \Psi_\Eps(s, x, v) \dv\dx\ds}  \nn
\\
& \leq
   \Eps C_1(\Omega) \norm{h_2}_{C^1(\Omega)}
    \int_{\Omega} \abs{\Psi_\Eps (s, x, v)} \dv\dx\ds \nn
\\
& \leq
   \Eps C_2(\Omega) \norm{h_2}_{C^1(\Omega)}
   \norm{\Psi_\Eps}_{L^\infty(0, \infty; L^2(\R^n \times \Ss^{n-1}))}
\to 0 
\qquad \text{as $\Eps \to 0$} \,.
\end{align}
Combining~\eqref{limit:I-1-Eps} and~\eqref{limit:I-2-Eps}, we have
\begin{align} \label{limit:LHS}
    \text{LHS of~\eqref{def:soln-1}} 
\longrightarrow
   \Psi_0 \quad \text{in $\CalD'$.}
\end{align}
The limiting equation~\eqref{eq:Psi-0} hence follows from~\eqref{limit:RHS} and~\eqref{limit:LHS}. 
\end{proof}

Next, we study the convergence of averages of $\Psi_\Eps$. To this end, we apply the Fourier transform in $x$ to \eqref{def:soln} for $a.e.$ $s, v$. This gives
\begin{align} \label{F-1}
   \hat\Psi_\Eps(s, \xi, v')
   = \vpran{\int_0^\infty \!\!\!\! \int_{\Ss^{n-1}} 
         \vpran{\sigma(v' \cdot \bar v) - \theta(\Eps) (1-c)} p(\tau)
         \hat\Psi_\Eps(\tau, \xi, \bar v) {\rm d}\bar v\dtau} e^{-i \Eps v' \cdot \xi s}
       + \theta(\Eps) \hat Q(\xi, v') e^{-i \Eps v' \cdot \xi s}\,,
\end{align}
where $\xi$ is the Fourier variable and $\hat u$ denotes the Fourier transform in $x$ of $u$. The free velocity variable is changed from $v$ to $v'$ for later notational convenience. Note that switching the order of integration on the right-hand side of \eqref{def:soln} when applying the Fourier transform is valid. Indeed, denote
\begin{align*}
   w_1(v') 
&  = \int_{\R^n} \vpran{\int_0^\infty\int_{\Ss^{n-1}} 
         \sigma(v' \cdot \bar v) p(\tau)
         \Psi_\Eps(\tau, x, \bar v) {\rm d} \bar v\dtau} e^{-i x \cdot \xi} \dx \,,
\\
   w_2(v')
& = \int_0^\infty\int_{\Ss^{n-1}}
         \sigma(v' \cdot \bar v) p(\tau)
         \vpran{\int_{\R^n}\Psi_\Eps(\tau, x, \bar v) e^{-i x \cdot \xi} \dx} {\rm d} \bar v\dtau \,.
\end{align*}
By \eqref{bound:CS-1} with $s=0$, we have $w_1, w_2 \in L^1(\Ss^{n-1})$. Moreover,  for any function $\phi_1 \in L^\infty(\Ss^{n-1})$, 
\begin{align*}
& \quad \,
  \int_{\Ss^{n-1}}\int_{\R^n} \int_0^\infty\int_{\Ss^{n-1}} 
         \sigma(v' \cdot \bar v) p(\tau)
         \abs{\Psi_\Eps(\tau, x, \bar v) \phi_1(v')} \dv'\dtau \dx\dvbar 
\\
&\leq 
   \norm{\phi_1}_{L^\infty(\Ss^{n-1})}
   \norm{\Psi_\Eps}_{L^\infty((0, \infty); L^1(\R^n \times \Ss^{n-1}))}
  < \infty \,.
\end{align*}
Hence, by Fubini's theorem, it holds that
\begin{align*}
    \int_{\Ss^{n-1}} w_1(v') \phi_1(v') \dv'
  = \int_{\Ss^{n-1}} w_2(v') \phi_1(v') \dv'
\end{align*}
for any $\phi_1 \in L^\infty(\Ss^{n-1})$. This shows $w_1 = w_2$ and \eqref{F-1} is valid.

Hinted by \eqref{F-1}, we consider the following averaged quantity of $\hat\Psi_\Eps$:
\begin{align} \label{def:phi-Eps}
   \hat\phi_\Eps(\xi, v)
   =  \int_0^\infty \int_{\Ss^{n-1}}
          \sigma(v \cdot \bar v) p(s) \hat \Psi_\Eps(s, \xi, \bar v) \dvbar\ds \,.
\end{align}
We will first show that the limit of the velocity average of $\hat\phi_\Eps$ satisfies a diffusion equation. The equation for $\hat\phi_\Eps$ is derived by multiplying $\sigma(v \cdot v') p(s)$ to equation~\eqref{F-1} and integrating in $s, v'$. It has the form
\begin{align} \label{eq:phi-Eps}
   \hat\phi_\Eps(\xi, v)
   = & \int_0^\infty \int_{\Ss^{n-1}} 
          \sigma(v \cdot v') p(s) \hat\phi_\Eps(\xi, v') 
          e^{-i \Eps v' \cdot \xi s} \dv'\ds \nn
\\
    & - \theta(\Eps) (1-c) \int_0^\infty \int_{\Ss^{n-1}}           
          \int_0^\infty \int_{\Ss^{n-1}}  
             \sigma(v \cdot v') p(s) p(\tau) \hat \Psi_\Eps(\tau, \xi, \bar v)  
          e^{-i \Eps v' \cdot \xi s} \dvbar \dtau\dv'\ds
\\
     &  + \theta(\Eps)
            \int_0^\infty \int_{\Ss^{n-1}} 
          \sigma(v \cdot v') p(s) \hat Q(\xi, v') 
          e^{-i \Eps v' \cdot \xi s} \dv'\ds \,. \nn
\end{align}
For the ease of notation, we introduce the operator $\CalK_\Eps$ and the terms $\CalA_\Eps \hat\Psi_\Eps, \hat q_\Eps$ as
\begin{align} \label{def:K-Eps}
   \CalK_\Eps \hat u 
&    =  \int_0^\infty \int_{\Ss^{n-1}} 
          \sigma(v \cdot v') p(s) \hat u(\xi, v') 
          e^{-i \Eps v' \cdot \xi s} \dv'\ds \,,
\qquad
   \hat q_\Eps = \CalK_\Eps \hat Q \,.
\\
   \CalA_\Eps \hat\Psi_\Eps
&   = \int_0^\infty \int_{\Ss^{n-1}}           
          \int_0^\infty \int_{\Ss^{n-1}}  
             \sigma(v \cdot v') p(s) p(\tau) \hat\Psi_\Eps(\tau, \xi, \bar v)
           e^{-i \Eps v' \cdot \xi s} \dvbar \dtau\dv'\ds \,.
\end{align}
Then \eqref{eq:phi-Eps} becomes
\begin{align} \label{eq:phi-Eps-1}
    \hat\phi_\Eps 
    = \CalK_\Eps \hat\phi_\Eps  
       - \theta(\Eps) (1-c) \CalA_\Eps \hat\Psi_\Eps
       + \theta(\Eps) \hat q_\Eps \,.
\end{align}
We can further write \eqref{eq:phi-Eps-1} as
\begin{align} \label{eq:phi-Eps-2}
   \frac{1}{\theta(\Eps)} (\hat\phi_\Eps - \CalK_\Eps \hat\phi_\Eps)
   + (1-c) \CalA_\Eps \hat\Psi_\Eps 
   = \hat q_\Eps \,.
\end{align}
In order to pass to the limit in $\Eps$ in equation~\eqref{eq:phi-Eps-2}, we need uniform bounds on $\CalK_\Eps \hat\phi_\Eps$, $\CalA_\Eps \hat\phi_\Eps$, and $\hat q_\Eps$. This is stated in the following lemma:
\begin{lem} \label{lem:phi-Eps}
Suppose $\sigma, c, p$ satisfy the conditions in Theorem~\ref{thm:well-posedness}. Suppose $Q \in  L^1\cap L^2(\R^n \times \Ss^{n-1})$ and $\hat\Psi_\Eps \in L^\infty(0, \infty; L^1 \cap L^2(\R^n \times \Ss^{n-1}))$ is the solution to equation~\eqref{def:soln}. Then $\hat\phi_\Eps \in L^2(\R^n \times \Ss^{n-1})$. Moreover,  there exists a constant $C_1>0$ which only depends on $Q, c$ such that
\begin{align*}
    \norm{\hat\phi_\Eps}_{L^2(\R^n \times \Ss^{n-1})} \leq C_1 \,,
\quad
    \norm{\CalK_\Eps \hat\phi_\Eps}_{L^2(\R^n \times \Ss^{n-1})} \leq C_1  \,,
\quad
    \norm{\CalA_\Eps \hat\Psi_\Eps}_{L^2(\R^n \times \Ss^{n-1})} \leq C_1  \,,
\quad
     \norm{\hat q_\Eps}_{L^2(\R^n \times \Ss^{n-1})} \leq C_1 \,.
\end{align*}
\end{lem}
\begin{proof}
These are all direct consequences of the inequality~\eqref{bound:CS-1} in Lemma~\ref{lem:CS} and Parseval's identity.
\end{proof}

We can now derive the limit of $\hat\phi_\Eps$ along a subsequence. 
\begin{lem} \label{lem:converg-phi}
Suppose $Q \in  L^1\cap L^2(\R^n \times \Ss^{n-1})$ and $\hat\Psi_\Eps \in L^\infty(0, \infty; L^1 \cap L^2(\R^n \times \Ss^{n-1}))$ is the solution to equation~\eqref{def:soln}. Let $\Psi_0$ be the limit defined in Theorem~\ref{thm:Psi}.
Then there exists a subsequence $\hat\phi_{\Eps_k}$ such that $\hat\phi_{\Eps_k} \to \hat\Psi_0$ weakly in $L^2(\R^n \times \Ss^{n-1})$. As a consequence, $\vint{\hat\phi_{\Eps_k}} \to \hat \Psi_0$ weakly in $L^2(\R^n)$.
\end{lem}
\begin{proof}
The proof follows from the uniform bounded of $\hat\phi_\Eps$ and a similar argument as in~\eqref{converg:h-1} (by replacing $\Psi_\Eps$ in ~\eqref{converg:h-1} with $\hat\Psi_\Eps$).
\end{proof}

Define the operators $\CalL_\Eps, \CalK, \CalL$ as
\begin{align} \label{def:K-L}
    \CalK \hat u= \int_{\Ss^{n-1}} \sigma(v \cdot v') \hat u(v') \dv' \,,
\qquad
   \CalL_\Eps = \ii - \CalK_\Eps \,,
\qquad
   \CalL = \ii - \CalK \,.
\end{align}
where $\ii$ is the identity operator. 
Then \eqref{eq:phi-Eps-1} can be further reformulated as 
\begin{align} \label{eq:phi-Eps-4}
    \frac{1}{\theta(\Eps)}\CalL \hat\phi_\Eps 
    - \int_0^\infty \int_{\Ss^{n-1}}
          \sigma(v \cdot v') p(s) \hat\phi_\Eps(\xi, v') 
          \frac{e^{-i \Eps v' \cdot \xi s} - 1}{\theta(\Eps)}
      \dv'\ds
 = - (1-c) \CalA_\Eps \hat\Psi_\Eps
    + \hat q_\Eps \,. 
\end{align}

We summarize some properties of $\CalL$ in the following lemma:
\begin{lem} \label{lem:CalL}
Let $\CalL, \CalK$ be defined as in~\eqref{def:K-L}. Then
\begin{enumerate}
\item $\CalK: L^2(\Ss^{n-1}) \to L^2(\Ss^{n-1})$ is compact and $\CalL$ is Fredholm. 

\item $\NullL = \Span\{1\}$.

\item Fix $\xi \in \R^n$. Then $v \cdot \xi \in (\NullL)^\perp$ is an eigenfunction of $\CalL$ with eigenvalue $1-\bar\mu_0$. Here, 
$$
\bar\mu_0 =  \frac{1}{2} \int_{-1}^1 \mu \sigma(\mu) \dmu < 1
$$ 
is the average scattering cosine.
\end{enumerate}
\end{lem}

\begin{proof}
Part (1) and (2) are classical results regarding transport equations \cite{BarSanSen84}. Part (3) follows from direct calculation. Indeed, 
by a symmetry argument
\begin{align}
  \CalK (v \cdot \xi)
= \int_{\Ss^{n-1}} \sigma(v \cdot v') (v' \cdot \xi) \dv'
= \vpran{\int_{\Ss^{n-1}} \sigma(v \cdot v') (v' \cdot v) \dv'} (v \cdot \xi)
= \vpran{\frac{1}{2} \int_{-1}^1 \mu \sigma(\mu) \dmu} (v \cdot \xi) \,.
\end{align}
Therefore $v \cdot \xi$ is an eigenfunction of $\CalL$ with the associated eigenvalue $1 - \bar\mu_0$.
\end{proof}


In the rest of the section we will find the macroscopic equation that $\hat \Psi_0$ satisfies by passing to the limit in \eqref{eq:phi-Eps-4}. 
The main result builds upon the key estimates summarized in the following proposition:
\begin{prop} \label{prop:Lambda-Eps}
Suppose $\kappa(\mu) \geq 0$ and satisfies that $\frac{1}{2}\int_{-1}^1 \kappa(\mu) \dmu = 1$.
Suppose $p$ satisfies  that
\begin{align} \label{assump:p}
   \int_0^\infty p(s) \ds = 1, \qquad \int_0^\infty s p(s) \ds < \infty \,,
\qquad p(s) = \frac{d_0}{s^{\alpha + 1}} 
\quad 
  \text{for $s > 1$,}
\end{align}
where $d_0 > 0$ is a constant. 
For any given $v' \in \Ss^{n-1}$ and $\xi \in \R^n$, define $\Lambda_\Eps$ as
\begin{align*}
    \Lambda_\Eps(\xi, v') = \int_{\Ss^{n-1}} \int_0^\infty \kappa(v' \cdot v)  p(s) \frac{\cos(\Eps v \cdot \xi s) - 1}{\theta(\Eps)} \ds\dv \,.
\end{align*}
Then there exists a generic constant $c_0 > 0$ which is independent of $\Eps, v', \xi$ such that 
\begin{itemize}
\item[(a1)] if $\alpha > 2$ and we choose $\theta(\Eps) = \Eps^2$, then $ |\Lambda_\Eps| \leq c_0 |\xi|^2$;
\item[(b1)] if $1 < \alpha < 2$ and we choose $\theta(\Eps) = \Eps^\alpha$, then $ |\Lambda_\Eps| \leq c_0 |\xi|^\alpha$;
\item[(c1)] if $\alpha = 2$ and we choose $\theta(\Eps) = -\Eps^2 \ln \Eps$, then $ |\Lambda_\Eps| \leq c_0 |\xi|^2$.
\end{itemize}
Moreover, there exist $\tilde D_1, \tilde D_2, \tilde D_3> 0$ (that may depend on $\xi, v'$) such that
\begin{itemize}
\item[(a2)] if $\alpha > 2$ and we choose $\theta(\Eps) = \Eps^2$, then $\lim_{\Eps \to 0} \Lambda_\Eps = - \tilde D_1 |\xi|^2$;
\item[(b2)] if $1 < \alpha < 2$ and we choose $\theta(\Eps) = \Eps^\alpha$, then $\lim_{\Eps \to 0} \Lambda_\Eps = - \tilde D_2 |\xi|^\alpha$;
\item[(c2)] if $\alpha = 2$ and we choose $\theta(\Eps) = -\Eps^2 \ln \Eps$, then $\lim_{\Eps \to 0} \Lambda_\Eps = - \tilde D_3 |\xi|^2$.
\end{itemize}
\noindent Here the convergence statements are to be understood pointwise in $\xi, v'$. In the special case where $\kappa(v' \cdot v)$ is a constant, the coefficients $\tilde D_1, \tilde D_2, \tilde D_3$ are all constants that are independent of $\xi, v'$.
\end{prop}
\begin{proof}
We will show these estimates case by case. 

\noindent (a) Suppose $\alpha > 2$ and let $\theta(\Eps) = \Eps^2$. This is the small-tail case where 
\begin{align} \label{cond:p-1}
    D_0 \Denote \int_0^\infty p(s) s^2 \ds < \infty \,.
\end{align}
In this case, for each fixed $v' \in \Ss^{n-1}$ and $\xi \in \R^n$, the integrand of $\Lambda_\Eps$ satisfies 
\begin{align} \label{bound:integrand-1}
   0 \leq \kappa(v \cdot v') p(s) \frac{\abs{\cos(\Eps v \cdot \xi s) - 1}}{\theta(\Eps)}
\leq
   2|\xi|^2 s^2 p(s) \kappa(v \cdot v') \in L^1((0, \infty) \times \Ss^{n-1})\,.
\end{align}
Therefore, 
\begin{align*}
 |\Lambda_\Eps| 
 = \abs{\int_{\Ss^{n-1}} \int_0^\infty \kappa(v \cdot v') p(s) \frac{\cos(\Eps v \cdot \xi s) -1}{\Eps^2}\ds\dv}
\leq 
   2|\xi|^2 \int_{\Ss^{n-1}} \int_0^\infty
        \kappa(v \cdot v') s^2 p(s) \ds\dv 
= 2 |\xi|^2 D_0 \,.
\end{align*}
In addition, by Lebesgue's Dominated Convergence Theorem, we have
\begin{align} \label{limit:Lambda-Eps-1}
   \lim_{\Eps \to 0} \Lambda_\Eps
&= -2 \int_{\Ss^{n-1}} \int_0^\infty \kappa(v \cdot v') p(s)
          \lim_{\Eps \to 0} \vpran{\frac{\sin^2 \vpran{\frac{\Eps v \cdot \xi s} {2}}}{\theta(\Eps)}} \ds\dv \nn
\\
&= -\frac{|\xi|^2}{2} \vpran{\int_{\Ss^{n-1}} \kappa(v \cdot v') \vpran{v \cdot e_{\xi}}^2 \dv} D_0
= -\tilde D_1 |\xi|^2 \,,
\qquad
   e_\xi = \xi/|\xi| \,,
\end{align}
where $\tilde D_1$ is defined as
\begin{align} \label{def:D-1}
   \tilde D_1(v', e_\xi) = \frac{1}{2} \vpran{\int_{\Ss^{n-1}}\kappa(v \cdot v')  \vpran{v \cdot e_{\xi}}^2 \dv} D_0  \,,
\qquad
   e_\xi = \xi/|\xi| \,.   
\end{align}
By its definition, it is clear that $\tilde D_1$ is independent of $v, \xi$ if $\kappa$ is a constant. 

\noindent{(b)} Now suppose $1 < \alpha < 2$. In this case we take $\theta(\Eps) = \Eps^\alpha$.  Break the integration domain in $\Lambda_\Eps$ into $s >1$ and $s < 1$ and let 
\begin{align*}
   \Lambda_{\Eps,1}
= - 2 \int_{\Ss^{n-1}}\int_{s>1} \kappa\vpran{v \cdot v'}  \, p(s) \frac{\sin^2 \vpran{\frac{\Eps  |\xi| s \vpran{v \cdot e_\xi}} {2}}}{\Eps^\alpha} \ds\dv \,,
\\
   \Lambda_{\Eps,2}
= - 2 \int_{\Ss^{n-1}} \int_{s<1} \kappa\vpran{v \cdot v'}  \, p(s) \frac{\sin^2 \vpran{\frac{\Eps  |\xi| s \vpran{v \cdot e_\xi}} {2}}}{\Eps^\alpha} \ds\dv \,.
\end{align*}
The term $\Lambda_{\Eps, 2}$ is bounded as 
\begin{align} \label{bound:Lambda-2-2}
   \abs{\Lambda_{\Eps,2}}
&\leq
   \frac{|\xi|^2 \, \Eps^{2-\alpha}}{2} \int_{\Ss^{n-1}} \int_{s < 1} \kappa\vpran{v \cdot v'} \, p(s) s^2 \ds\dv \nn
\\
&\leq   \frac{|\xi|^2 \Eps^{2-\alpha}}{2} \int_{\Ss^{n-1}} \int_0^\infty \kappa\vpran{v \cdot v'} \, p(s) \ds\dv
= \frac{1}{2} |\xi|^2 \Eps^{2-\alpha} \,.
\end{align}
Since $1 < \alpha < 2$, we also have 
\begin{align} \label{converg:Lambda-Eps-2-2}
    \Lambda_{\Eps,2} \to 0  
\qquad
   \text{as $\Eps \to 0$ for each fixed $\xi$.}
\end{align}

\Ni To derive the bound of $\Lambda_{\Eps, 1}$, we apply the change of variable $\tau = \Eps |\xi| s$. By the tail behaviour of $p(s)$ in \eqref{assump:p},
\begin{align*}
   \Lambda_{\Eps,1}
= - 2 d_0 |\xi|^\alpha \int_{\Ss^{n-1}} \int_{\tau >\Eps |\xi|} 
     \kappa\vpran{v \cdot v'} \,\frac{\sin^2 \vpran{\frac{\tau \vpran{v \cdot e_\xi}} {2}}}{\tau^{\alpha + 1}}  \dtau\dv
\end{align*}
Since $1 < \alpha < 2$, the integrand in the last term of the above equation satisfies that
\begin{align} \label{def:D-2}
 \tilde   D_2 
&\Denote 
   2 d_0\int_{\Ss^{n-1}} \int_0^\infty \kappa(v \cdot v') \frac{\sin^2 \vpran{\frac{\tau \vpran{v \cdot e_\xi}} {2}}}{\tau^{\alpha + 1}} \dtau\dv \nn
\\
& \leq 
   2 d_0 \int_1^\infty \frac{1}{\tau^{\alpha+1}} \dtau
   + \frac{d_0}{2} \int_0^1 \frac{1}{\tau^{\alpha-1}} \dtau
=  \frac{2d_0}{\alpha} + \frac{d_0}{2(2-\alpha)}
< \infty \,.
\end{align}
Hence, we have
\begin{align} \label{bound:Lambda-Eps-1-2}
   \abs{\Lambda_{\Eps, 1}} 
\leq \tilde D_2 |\xi|^\alpha
\leq \vpran{\frac{2}{\alpha} + \frac{1}{2(2-\alpha)}} |\xi|^\alpha \,,
\end{align}
Furthermore, we can apply Lebesgue's Dominated Convergence theorem and get that for each $\xi \in \R^n$,
\begin{align} \label{converg:Lambda-Eps-1-2}
     \Lambda_{\Eps, 1} \to -\tilde D_2 |\xi|^\alpha 
\end{align}
Combining \eqref{converg:Lambda-Eps-2-2} with \eqref{converg:Lambda-Eps-1-2}, we derive that 
\begin{align*}
    |\Lambda_\Eps| \leq \vpran{\frac{1}{2} + \frac{2}{\alpha} + \frac{1}{2(2-\alpha)}} |\xi|^\alpha \,,
\qquad
   \Lambda_\Eps \to -\tilde D_2 |\xi|^\alpha \quad \text{ for each $\xi \in \R^n$} \,.
\end{align*}
where $\tilde D_2$ is defined in \eqref{def:D-2}. It is also clear from its definition in~\eqref{def:D-2} that $\tilde D_2$ is a constant (independent of $\xi$) if $\kappa$ is a constant.

\smallskip
\noindent (c) In the borderline case where $\alpha=2$, we choose $\theta(\Eps) = \Eps^2 \ln (1/\Eps)$. The choice of $\theta$ is slightly less obvious than the previous two cases but it will be clear from the estimates below. 

We again split the integration domain into $s < 1$ and $s > 1$ and apply the change of variable $\tau = \Eps |\xi| s$ in the subdomain $s > 1$ . Define
\begin{align*}
   \Lambda_{\Eps,3}
= - \frac{2 d_0 |\xi|^2}{\abs{\ln \Eps}} \int_{\Ss^{n-1}}\int_{\tau>\Eps |\xi|} \kappa(v \cdot v')\frac{\sin^2 \vpran{\frac{\tau \vpran{v \cdot e_\xi}} {2}}}{\tau^{3}} \dtau\dv \,,
\\
   \Lambda_{\Eps,4}
= - 2 \int_{\Ss^{n-1}} \int_{s<1} \kappa(v \cdot v') p(s) \frac{\sin^2 \vpran{\frac{\Eps |\xi| s \vpran{v \cdot e_\xi}} {2}}}{\Eps^2 \abs{\ln \Eps}} \ds\dv \,.
\end{align*}
First we bound $\Lambda_{\Eps,4}$ as 
\begin{align*}
   |\Lambda_{\Eps,4}| &= 2 \int_{\Ss^{n-1}} \int_{s<1} \kappa(v \cdot v') p(s) \frac{\sin^2 \vpran{\frac{\Eps  |\xi| s \vpran{v \cdot e_\xi}} {2}}}{\Eps^2 \ln(1/\Eps)} \ds\dv
\leq 
  \frac{2|\xi|^2}{\abs{\ln \Eps}} \int_{\Ss^{n-1}} \int_{s<1}
   \kappa(v \cdot v') p(s) \ds\dv
\leq \frac{2 |\xi|^2}{\abs{\ln \Eps}}. 
\end{align*}
This shows
\begin{align} \label{converg:Lambda-Eps-4}
    \Lambda_{\Eps,4} \to 0  
\qquad
   \text{as $\Eps \to 0$ for each fixed $\xi$.}
\end{align}
Next, we separate $\Lambda_{\Eps, 3}$ as 
\begin{align*}
   \Lambda_{\Eps,3}
&= \frac{2 |\xi|^2}{\abs{\ln \Eps}} \int_{\Ss^{n-1}}\int_{\Eps |\xi|}^1 \kappa(v \cdot v') \frac{\sin^2 \vpran{\frac{\tau \vpran{v \cdot e_\xi}} {2}}}{\tau^3} \dtau\dv
   +\frac{2 |\xi|^2}{\abs{\ln \Eps}} \int_{\Ss^{n-1}}\int_{1}^\infty \kappa(v \cdot v') \frac{\sin^2 \vpran{\frac{\tau \vpran{v \cdot e_\xi}} {2}}}{\tau^3} \dtau\dv
\\
&\Denote \Lambda_{\Eps, 3, 1} + \Lambda_{\Eps, 3, 2} \,.
\end{align*}
The second term $\Lambda_{\Eps, 3, 2}$ is bounded as
\begin{align*}
    \abs{\Lambda_{\Eps, 3, 2}}
    = \frac{c_1 |\xi|^2}{\abs{\ln \Eps}}
    \leq \frac{2 |\xi|^2}{\abs{\ln \Eps}} \,,
\qquad
    \text{since} \quad c_1 =  \int_{\Ss^{n-1}}\int_{1}^\infty \kappa(v \cdot v') \frac{\sin^2 \vpran{\frac{\tau \vpran{v \cdot e_\xi}} {2}}}{\tau^3} \dtau\dv \leq 2 \,.
\end{align*}
Therefore, $\Lambda_{\Eps, 3, 2}$ is bounded and 
\begin{align} \label{converg:Lambda-Eps-3-2}
    \Lambda_{\Eps,3,2} \to 0  
\qquad
   \text{as $\Eps \to 0$ for each fixed $\xi$.}
\end{align}
We are left to show the bound and limit of $\Lambda_{\Eps, 3, 1}$. 
Denote 
\begin{align*}
   h(\tau) = \int_{\Ss^{n-1}} \kappa(v \cdot v') \sin^2(\frac{\tau \vpran{v \cdot e_\xi}}{2}) \dv \,.
\end{align*}
Then by L'H\^{o}pital's rule,
\begin{align*}
   \lim_{\Eps \to 0} \frac{2}{\abs{\ln \Eps}} \int_{\Ss^{n-1}}\int_{\Eps |\xi|}^1 \kappa(v \cdot v')\frac{\sin^2 \vpran{\frac{\tau \vpran{v \cdot e_\xi}} {2}}}{\tau^{3}} \dtau\dv
&= \lim_{\Eps \to 0} \frac{2}{\ln (1/\Eps)} \int_{\Eps |\xi|}^1 \frac{h(\tau)}{|\tau|^{3}}\dtau
\\
&= \lim_{\Eps \to 0} \vpran{2\Eps h(1) - \frac{2h(\Eps |\xi|)}{\Eps^2 |\xi|^2}}
\Denote - \tilde D_3 \,,
\end{align*}
where
\begin{align} \label{def:D-3}
   \tilde D_3 = \lim_{\Eps \to 0} \frac{2h(\Eps |\xi|)}{\Eps^2 |\xi|^2}
         = \frac{1}{2} \int_{\Ss^{n-1}} \kappa(v \cdot v') \vpran{v \cdot e_\xi}^2 \dv \,.
\end{align}
By the bounds and limits for $\Lambda_{\Eps, 3, 1}$, $\Lambda_{\Eps, 3, 2}$, and $\Lambda_{\Eps, 4}$, we conclude that there exists a constant $c_0 > 0$ which is independent of $\Eps, v', \xi$ such that
\begin{align*}
    |\Lambda_\Eps| \leq c_0 |\xi|^2 \,,
\qquad
   \Lambda_\Eps \to \tilde D_3 |\xi|^2 \quad \text{ for each $\xi \in \R^n$} \,,
\end{align*}
where $\tilde D_3$ is defined in \eqref{def:D-3}. It is also clear from its definition in~\eqref{def:D-3} that $\tilde D_3$ is a constant (independent of $\xi$) if $\kappa$ is a constant. 
\end{proof}

Using Proposition~\ref{prop:Lambda-Eps}, we can now state our main theorem in more detail and show its proof.
\begin{thm}
Suppose the scattering coefficient $c$ and the cross section $\sigma$ satisfies that
\begin{align} \label{assump:sigma}
   0 < c < 1 \,,
\qquad
    \int_{\Ss^{n-1}} \sigma(v \cdot v') \dv = 1\,,
\qquad
   \sigma(v\cdot v') \geq \sigma_0 > 0
\end{align}
for some $\sigma_0>0$. Suppose the path-length distribution function $p$ satisfies \eqref{assump:p} and $\Psi_\Eps = \Psi_\Eps e^{-\int_0^s \Sigma_t(\tau) \dtau}$ is the solution to~\eqref{eq2}.
Let $\phi_\Eps, q_\Eps$ be the functions defined in \eqref{def:phi-Eps} and \eqref{def:K-Eps}. ~Then 
\begin{enumerate}
\item[(a)] there exists $\Psi_0 \in L^2(\R^n)$ such that
\begin{align*}
    \Psi_\Eps \to \Psi_0
\qquad in \,\,
   w^\ast-L^\infty(0, \infty; L^2(\R^n \times \Ss^{n-1})).
\end{align*}
\item[(b)]
there exists $ q \in L^2(\R^n \times \Ss^{n-1})$ such that  
\begin{align*}
   \phi_{\Eps} \to \Psi_0 \,,
\qquad
   q_\Eps \to q
\qquad
   \text{weakly in $L^2(\R^n \times \Ss^{n-1})$}.
\end{align*}
Moreover, with the choices of $\theta(\Eps)$ in Proposition~\ref{prop:Lambda-Eps}, the limit $\Psi_0$ satisfies
\begin{itemize}
\item[(b1)]  $ D_1 (-\Delta) \Psi_0 + (1 - c) \Psi_0 = \int_{\Ss^{n-1}} Q(x, v) \dv$ if $\alpha > 2$;

\item[(b2)]  $D_2 (-\Delta)^{\alpha/2} \Psi_0 + (1 - c) \Psi_0 = \int_{\Ss^{n-1}} Q(x, v) \dv$ if $1 < \alpha < 2$;

\item[(b3)]  $ D_3 (-\Delta) \Psi_0 + (1 - c) \Psi_0 = \int_{\Ss^{n-1}} Q(x, v) \dv$ if $\alpha = 2$,

\end{itemize}
where $D_1, D_2, D_3$ are positive constants defined as 
\begin{align*}
   D_1 
   = \frac{1}{3} \frac{\vpran{\int_0^\infty p(s) s \ds}^2 \bar\mu_0}{1 - \bar\mu_0}
       + \frac{1}{6} \int_0^\infty p(s) s^2 \ds\,,
\end{align*}
and
\begin{align*}
   D_2 = 2 d_0\int_{\Ss^{n-1}} \int_0^\infty \frac{\sin^2 \vpran{\frac{\tau \vpran{v \cdot e_\xi}} {2}}}{\tau^{\alpha + 1}} \dtau\dv\,,
\qquad
   D_3 = \frac{1}{2} \int_{\Ss^{n-1}} \vpran{v \cdot e_\xi}^2 \dv
           = \frac{1}{6} \,.
\end{align*}
where $d_0$ is the constant in~\eqref{assump:p}.

\item[(c)] Let $\eta_\Eps(x) = \int_0^\infty \int_{\Ss^{n-1}} \Psi_\Eps(s, x, v) \dv\ds$. Then there exists $\eta_0 \in L^2(\R^n)$ such that
\begin{align*}
   \eta_\Eps \to \eta_0
\qquad
   \text{weakly in $L^2(\R^n)$.}
\end{align*}
Moreover, $\eta_0 = \beta_0 \Psi_0$ where the positive constant $\beta_0$ is defined in~\eqref{def:beta-0}. Therefore $\eta_0$ satisfies similar diffusion equations as in (b1)-(b2) with the source term replaced by $\beta_0 \int_{\Ss^{n-1}} Q(x, v) \dv$.

\end{enumerate}
\end{thm}
\begin{proof}
(a) The convergence along a subsequence is proved in Theroem~\ref{thm:Psi}.  The convergence of the full sequence will be clear from the proof of Part (b) and (c). 

\noindent (b) Integrating \eqref{eq:phi-Eps-4} in terms of $v$ to annihilate the singular term $\frac{1}{\theta(\Eps)} \CalL \hat\phi_\Eps$, we have
\begin{align*} 
    - \int_{\Ss^{n-1}} \int_{\Ss^{n-1}}
          \sigma(v \cdot v') \hat\phi_\Eps(\xi, v') 
          w_{\Eps}(\xi, v')
      \dv'
 = - (1-c) \int_{\Ss^{n-1}}\CalA_\Eps \hat\Psi_\Eps \dv
    + \int_{\Ss^{n-1}} \hat q_\Eps(\xi, v) \dv \,,
\end{align*}
where $w_\Eps$ is defined as
\begin{align} \label{def:w-Eps}
    \hat w_\Eps(\xi, v) 
    = \int_0^\infty p(s)\frac{e^{-i \Eps v \cdot \xi s} - 1}{\theta(\Eps)} \ds \,,
\end{align}
By the assumption for $\sigma$ in \eqref{assump:sigma}, the above equation can be written as 
\begin{align} \label{eq:phi-Eps-5}
    - \int_{\Ss^{n-1}}
         \hat\phi_\Eps(\xi, v') w_{\Eps}(\xi, v')
      \dv'\dv
 = - (1-c) \int_{\Ss^{n-1}}\CalA_\Eps \hat\Psi_\Eps \dv
    + \int_{\Ss^{n-1}} \hat q_\Eps(\xi, v) \dv \,.
\end{align}
The eventual diffusion equation will be obtained by passing to the limit along the subsequence $\hat\phi_{\Eps_k}$ in \eqref{eq:phi-Eps-5}. 
We study the limit of each term in \eqref{eq:phi-Eps-5} along the subsequence $\hat\phi_{\Eps_k}$ given in Lemma~\ref{lem:converg-phi}. Up to a further subsequence and an abuse of notation, suppose $\hat q_{\Eps_k} \to \hat q$ weakly in $L^2(\R^n \times \Ss^{n-1})$. Then
\begin{align} \label{limit:RHS-2-1}
      \int_{\Ss^{n-1}} \hat q_{\Eps_k}(\xi, v) \dv
\to
    \int_{\Ss^{n-1}} \hat q \dv
\quad
   \text{weakly in $L^2(\R^n)$} \,.
\end{align}
By the definition of $\hat q_\Eps$ and Lebesgue Dominated Convergence Theorem, we have $\hat q = \hat Q$ and 
\begin{align} \label{limit:RHS-2}
      \int_{\Ss^{n-1}} \hat q_{\Eps_k}(\xi, v) \dv
\to
    \int_{\Ss^{n-1}} \hat Q \dv
\quad
   \text{weakly in $L^2(\R^n)$} \,.
\end{align}

Next, let $\hat g \in C_c(\R^n)$ be arbitrary. Then by Fubini's theorem,
\begin{align*}
& \quad \,
    \int_{\R^n} \hat g(\xi)  \vpran{ \int_{\Ss^{n-1}}
      \CalA_{\Eps_k} \hat \Psi_{\Eps_k} \dv} \dxi
\\
&= \int_{\R^n} \int_{\Ss^{n-1}} 
      \int_0^\infty \int_{\Ss^{n-1}}           
      \int_0^\infty \int_{\Ss^{n-1}}  
        \hat g(\xi)  \sigma(v \cdot v') p(s) p(\tau) \hat \Psi_\Eps(\tau, \xi, \bar v)
           e^{-i \Eps v' \cdot \xi s} \dvbar \dtau\dv'\ds\dv\dxi
\\
& 
  =  \int_{\R^n} \int_{\Ss^{n-1}} 
      \int_0^\infty \int_{\Ss^{n-1}}           
      \int_0^\infty \int_{\Ss^{n-1}}  
        \hat g(\xi)  \sigma(v \cdot v') p(s) p(\tau) \hat \Psi_\Eps(\tau, \xi, \bar v)
           \vpran{e^{-i \Eps v' \cdot \xi s} -1} \dvbar \dtau\dv'\ds\dv\dxi
\\
& \hspace{0.5cm}
   +  \int_{\R^n} \int_{\Ss^{n-1}} 
      \int_0^\infty \int_{\Ss^{n-1}}           
      \int_0^\infty \int_{\Ss^{n-1}}  
        \hat g(\xi)  \sigma(v \cdot v') p(s) p(\tau) \hat \Psi_\Eps(\tau, \xi, \bar v)
        \dvbar \dtau\dv'\ds\dv\dxi
\\
& 
   \to 
   \int_{\R^n} \int_{\Ss^{n-1}} \int_{\Ss^{n-1}} 
      \int_0^\infty \int_0^\infty
         \hat g(\xi) \sigma(v \cdot v') p(s) p(\tau) \hat\Psi_0(\xi)  \ds\dtau\dv' \dv\dxi
  = \int_{\R^{n-1}}
         \hat g(\xi) \hat\Psi_0(\xi)\dxi
\end{align*}
since $\int_0^\infty s p(s) \ds < \infty$. Hence,
\begin{align} \label{limit:RHS-1}
    \int_{\Ss^{n-1}}\CalA_{\Eps_k} \hat\Psi_{\Eps_k} \dv
\to 
  \hat \Psi_0
\quad
   \text{in the sense of distributions} \,.
\end{align}
Combining \eqref{limit:RHS-2} with \eqref{limit:RHS-1}, we obtain that along the subsequence $\Eps_k$,
\begin{align} \label{converg:RHS}
   \text{Right-hand side of \eqref{eq:phi-Eps-5}} 
\longrightarrow 
- (1-c) \, \hat\Psi_0 + \int_{\Ss^{n-1}} \hat Q(\xi, v) \dv 
\quad 
\text{in the sense of distributions.}
\end{align}
To find the limit of the left-hand side of \eqref{eq:phi-Eps-5} 
we rewrite it as
\begin{align*}
   \int_{\Ss^{n-1}}
         \hat\phi_\Eps(\xi, v') \hat w_{\Eps}(\xi, v') \dv'
& = \int_{\Ss^{n-1}}
       \vpran{\hat\phi_\Eps(\xi, v') - \vint{\hat\phi_\Eps}(\xi)} \hat w_{\Eps}(\xi, v')  \dv'
      + \int_{\Ss^{n-1}} 
            \vint{\hat\phi_\Eps}(\xi) \hat w_{\Eps}(\xi, v') \dv'
\\
& \Denote \EpsJ_{1} + \EpsJ_{2} \,,
\end{align*}
where we have introduced the notation
$$
\langle\cdot\rangle = \int \cdot\ dv
$$
for velocity averages. We find the limits of $\EpsJ_{1}$ and $\EpsJ_{2}$ separately.

\medskip
\Ni \underline{\em Limit of $\EpsJ_1$.} Recall the definition of $w_\Eps$ in \eqref{def:w-Eps} and rewrite $J_1^\Eps$ as
\begin{align*}
   J_1^\Eps
& = -i \int_{\Ss^{n-1}} \int_0^\infty
          \vpran{\hat\phi_\Eps(\xi, v') - \vint{\hat\phi_\Eps}(\xi)} 
          p(s) \frac{\sin(\Eps v' \cdot \xi s)}{\theta(\Eps)}\ds \dv'
\\
& \quad \,
      - \int_{\Ss^{n-1}} \int_0^\infty
          \vpran{\hat\phi_\Eps(\xi, v') - \vint{\hat\phi_\Eps}(\xi)} 
          p(s) \frac{1 - \cos(\Eps v' \cdot \xi s)}{\theta(\Eps)}\ds \dv'
\\
& = - i \frac{\Eps}{\theta(\Eps)} \int_{\Ss^{n-1}} \int_0^\infty
          \vpran{\hat\phi_\Eps(\xi, v') - \vint{\hat\phi_\Eps}(\xi)} 
          p(s) \vpran{v' \cdot \xi s} \ds \dv'
\\
& \quad \,
     - i \frac{\Eps}{\theta(\Eps)} \int_{\Ss^{n-1}} \int_0^\infty
          \vpran{\hat\phi_\Eps(\xi, v') - \vint{\hat\phi_\Eps}(\xi)} 
          p(s) \vpran{v' \cdot \xi s} \vpran{\frac{\sin(\Eps v' \cdot \xi s )}{\Eps v' \cdot \xi s} - 1}\ds \dv'
\\
& \quad \,
      - \int_{\Ss^{n-1}} \int_0^\infty
          \vpran{\hat\phi_\Eps(\xi, v') - \vint{\hat\phi_\Eps}(\xi)} 
          p(s) \frac{1 - \cos(\Eps v' \cdot \xi s)}{\theta(\Eps)}\ds \dv'
\\
& \Denote 
    J_{1,1}^\Eps + J_{1,2}^\Eps + J_{1,3}^\Eps \,.
\end{align*}
First we show that 
\begin{align} \label{limit:Eps-1-2-3}
   \EpsJ_{1,2} \to 0
\quad \text{and} \quad
   \EpsJ_{1,3} \to 0
\qquad
\text{in $\CalD'$}.
\end{align}
Indeed, we have the bounds
\begin{align*}
   \frac{\Eps}{\theta(\Eps)}
   \abs{\frac{\sin(\Eps v' \cdot \xi s )}{\Eps v \cdot \xi s} - 1}
\leq 
   \frac{1}{6}\frac{\Eps}{\theta(\Eps)}
   \vpran{\Eps v' \cdot \xi s}^2
\leq \frac{1}{6} \Eps |\xi|^2 \,.
\end{align*}
By the uniform bound of $\hat\phi_\Eps$ in $L^2(\R^n \times \Ss^{n-1})$, we have $\EpsJ_{1,2} \to 0$ in $L^2(\R^n) \times \Ss^{n-1}$. Moreover, for all the choices of $\theta(\Eps)$ we have
\begin{align*}
   \abs{\frac{1 - \cos(\Eps v' \cdot \xi s)}{\theta(\Eps)}}
\leq 
  \frac{1}{2} |\xi|^2 \,.
\end{align*}
Thus by $\hat\phi_\Eps - \vint{\hat\phi_\Eps} \to 0$ in $L^2(\R^n \times \Ss^{n-1})$, we have $\EpsJ_{1,3} \to 0$ in $\CalD'$. Therefore~\eqref{limit:Eps-1-2-3} holds. 

The limit of $\EpsJ_{1,1}$ is more involved. First, by Lemma~\ref{lem:CalL}, we rewrite $v' \cdot \xi$ as 
\begin{align} \label{def:nu-0}
    v' \cdot \xi = \nu_0 \CalL(v' \cdot \xi) \,,
\qquad
    \nu_0 = \frac{1}{1 - \bar\mu_0}  \,.
\end{align}
Then $J_{1,1}^\Eps$ becomes
\begin{align*}
   J_{1,1}^\Eps
&= - i 
     \frac{\nu_0 \, \Eps}{\theta(\Eps)} 
     \int_{\Ss^{n-1}} \int_0^\infty
          \vpran{\hat\phi_\Eps(\xi, v') - \vint{\hat\phi_\Eps}(\xi)} 
          p(s) \CalL\vpran{v' \cdot \xi} s \ds \dv'
\\
&= - i 
     \frac{\nu_0 \, \Eps}{\theta(\Eps)} 
     \int_{\Ss^{n-1}} \int_0^\infty
          \CalL\vpran{\hat\phi_\Eps(\xi, v') - \vint{\hat\phi_\Eps}(\xi)} 
          p(s) \vpran{v' \cdot \xi} s \ds \dv'
\\
&= - i 
     \frac{\nu_0 \, \Eps}{\theta(\Eps)} 
     \int_{\Ss^{n-1}} \int_0^\infty
          \CalL\hat\phi_\Eps(\xi, v') 
          p(s) \vpran{v' \cdot \xi} s \ds \dv' \,.
\end{align*}
By the equation for $\CalL \hat\phi_\Eps$ in \eqref{eq:phi-Eps-4},  we have
\begin{align*}
   \EpsJ_{1,1}
&   = -i \nu_0 \, \Eps
       \int_{\Ss^{n-1}} \vpran{\int_{\Ss^{n-1}}
          \sigma(v \cdot v') \hat\phi_\Eps(\xi, v) \hat w_\Eps(\xi, v)
      \dv} (v' \cdot \xi) \vpran{\int_0^\infty p(s) s \ds} \dv'
\\
 &  \quad \,   
   + i (1-c) \nu_0 \Eps \int_{\Ss^{n-1}} \CalA_\Eps \hat\Psi_\Eps(\xi, v') (v' \cdot \xi)\dv'
    - i \nu_0 \Eps  \int_{\Ss^{n-1}} \hat q_\Eps(\xi, v') (v' \cdot \xi)\dv'
\\
&\Denote \EpsJ_{1, 1,1} + \EpsJ_{1, 1,2} + \EpsJ_{1, 1,3} \,,
\end{align*}
where $\hat w_\Eps$ is definded in~\eqref{def:w-Eps}. 
By the uniform bounds of $\CalA_\Eps \hat\Psi_\Eps$ and $\hat q_\Eps$ in Lemma~\ref{lem:phi-Eps}, we have
\begin{align} \label{converg:J-1-2-3}
   \EpsJ_{1, 1, 2}, \, \EpsJ_{1, 1, 3} \to 0 
\quad
   \text{in the sense of distributions. }
\end{align}
To show the convergence of $\EpsJ_{1,1,1}$, we separate the cases where $\alpha > 2$ and $\alpha \leq 2$.  First, if $\alpha \leq 2$, then
\begin{align*}
   \EpsJ_{1,1,1}
&\leq 
   c_0 \vpran{1 + |\xi|^2} \frac{\Eps^2}{\theta(\Eps)}
  \abs{\int_{\Ss^{n-1}} \int_{\Ss^{n-1}} \sigma(v \cdot v') \hat\phi_\Eps(\xi, v') \dv\dv'}
\\
&\leq 
     c_0 \vpran{1 + |\xi|^2} \frac{\Eps^2}{\theta(\Eps)} \norm{\hat\phi_\Eps(\xi, \cdot)}_{L^2(\Ss^{n-1})} \,.
\end{align*}
In the case where $\alpha \leq 2$, we have 
\begin{align*}
    \frac{\Eps^2}{\theta(\Eps)} \to 0 
\qquad
   \text{as $\Eps \to 0$} \,.
\end{align*}
Therefore, if $\alpha \leq 2$, then
\begin{align*} 
    \EpsJ_{1,1,1} \to 0  
\qquad
   \text{ in the sense of distributions as $\Eps \to 0$.}
\end{align*}
Together with~\eqref{converg:J-1-2-3}, we have
\begin{align} \label{converg:EpsJ-1-1-1}
    \EpsJ_{1,1} \to 0 
\qquad
   \text{ in the sense of distributions as $\Eps \to 0$, }
\qquad
  \alpha \leq 2 \,.
\end{align}

In the case where $\alpha > 2$, we have $\theta(\Eps) = \Eps^2$. Separating the real and imaginary parts of $\EpsJ_{1,1}$, we get
\begin{align*}
   Re \vpran{\EpsJ_{1,1,1}}
= - \nu_0  \vpran{\int_0^\infty p(s) s \ds}
       \int_{\Ss^{n-1}} \int_{\Ss^{n-1}} \int_0^\infty
          \sigma(v \cdot v') \hat\phi_\Eps(\xi, v) 
          \vpran{v' \cdot \xi} p(s)
          \frac{\sin(\Eps v \cdot \xi s)}{\Eps}
      \ds \dv \dv'
\end{align*}
and
\begin{align*}
   Im \vpran{\EpsJ_{1,1,1}}
= - \nu_0  \vpran{\int_0^\infty p(s) s \ds}
       \int_{\Ss^{n-1}} \int_{\Ss^{n-1}} \int_0^\infty
          \sigma(v \cdot v') \hat\phi_\Eps(\xi, v) 
          \vpran{v' \cdot \xi} p(s)
          \frac{1 - \cos(\Eps v \cdot \xi s)}{\Eps}
      \ds \dv \dv' \,,
\end{align*}
where $Re\vpran{\EpsJ_{1,1,1}}$ and $Im\vpran{\EpsJ_{1,1,1}}$ are the real and imaginary parts of $\EpsJ_{1,1,1}$ respectively. Since $\hat\phi_{\Eps_k} \to \hat\Psi_0$ in $L^2(\R^n \times \Ss^{n-1})$, we have 
\begin{align} \label{limit:Re-EpsJ-1-1-1}
   Re \vpran{J_{1,1,1}^{\Eps_k}} 
\to 
   - \nu_0 \vpran{\int_0^\infty p(s) s \ds}^2 \hat \Psi_0
   \int_{\Ss^{n-1}} \int_{\Ss^{n-1}}
      \sigma(v \cdot v') \vpran{v \cdot \xi} \vpran{v' \cdot \xi} \dv'\dv
\quad
  \text{in $\CalD'$}
\end{align}
and
\begin{align} \label{limit:Im-EpsJ-1-1-1}
   Im \vpran{J_{1,1,1}^{\Eps_k}} \to 0
\quad
  \text{in $L^2(\R^n \times \Ss^{n-1})$} \,.
\end{align}
In the above convergences we have applied the bounds and limits
\begin{align*}
    \abs{\frac{\sin(\Eps v \cdot \xi s)}{\Eps}} 
\leq 
    |\xi| s \,,
\qquad
   \abs{\frac{1 - \cos(\Eps v \cdot \xi s)}{\Eps}}
\leq 
   \frac{1}{2} |\xi|^2 s^2 \Eps \,,
\qquad
  \frac{\sin(\Eps v \cdot \xi s)}{\Eps}
\to v \cdot \xi s  \quad \text{pointwise as $\Eps \to 0$} \,.
\end{align*}
The limit of $Re \vpran{\EpsJ_{1,1,1}}$ can be simplified as
\begin{align} \label{limit:Re-EpsJ-1-1-1-Simp}
& \quad \,
   - \nu_0 \vpran{\int_0^\infty p(s) s \ds}^2 \hat \Psi_0
   \int_{\Ss^{n-1}} \int_{\Ss^{n-1}}
      \sigma(v \cdot v') \vpran{v \cdot \xi} \vpran{v' \cdot \xi} \dv'\dv \nn
\\
&=    - \nu_0 \vpran{\int_0^\infty p(s) s \ds}^2 \hat \Psi_0
   \int_{\Ss^{n-1}} \vpran{\int_{\Ss^{n-1}}
      \sigma(v \cdot v') \vpran{v \cdot \xi} \dv} \vpran{v' \cdot \xi} \dv'
\\
&=    - \nu_0 \vpran{\int_0^\infty p(s) s \ds}^2 
      \vpran{\frac{1}{2} \int_{-1}^1 \mu \sigma(\mu) \dmu} \hat \Psi_0(\xi)
   \int_{\Ss^{n-1}}   \vpran{v' \cdot \xi}^2 \dv' \nn
\\
& = -\frac{1}{3} \nu_1 |\xi|^2 \hat \Psi_0(\xi) \,, \nn
\end{align}
where $\nu_0$ is defined in~\eqref{def:nu-0} and the constant $\nu_1$ is
\begin{align} \label{def:nu-1}
   \nu_1 
   =  \frac{\vpran{\int_0^\infty p(s) s \ds}^2
       \bar\mu_0}{1-\bar\mu_0} \,.
\end{align}
Combining~\eqref{converg:J-1-2-3}, \eqref{limit:Re-EpsJ-1-1-1}, \eqref{limit:Im-EpsJ-1-1-1}, and \eqref{limit:Re-EpsJ-1-1-1-Simp}, we have
\begin{align} \label{limit:EpsJ-1-1-2}
   J_{1,1}^{\Eps_k} \to  -\frac{1}{3} \nu_1 |\xi|^2 \hat \Psi_0(\xi)
\qquad
   \text{in $\CalD'$} \,,
\qquad
   \alpha > 2 \,.
\end{align}
As a summary, we have
\begin{align} \label{limit:EpsJ-1}
   J_1^{\Eps_k} \to 
   \begin{cases}
     0 \,, & \alpha \leq 2 \,, \\[2pt]
     -\frac{1}{3} \nu_1 |\xi|^2 \hat \Psi_0(\xi) \,,
     & \alpha > 2
   \end{cases}
\qquad
\text{in $\CalD'$ as $\Eps \to 0$.}
\end{align}

\medskip


\Ni \underline{\em Limit of $\EpsJ_2$.} To find the limit of $\EpsJ_2$, we make use of the symmetry of the integral and obtain that
\begin{align} \label{def:J-Eps-2}
    \EpsJ_{2} 
    = \vpran{ \int_{\Ss^{n-1}} \hat w_\Eps(\xi, v) \dv} \vint{\hat\phi_\Eps}(\xi)
    = \vpran{ \int_{\Ss^{n-1}} \int_0^\infty p(s) \frac{\cos(\Eps v \cdot \xi s) - 1}{\theta(\Eps)} \ds\dv} \vint{\hat\phi_\Eps}(\xi) \,.
\end{align}
Applying Proposition~\ref{prop:Lambda-Eps} with $\kappa(v \cdot v') = 1$ and the weak convergence of $\vint{\hat\phi_{\Eps_k}}$ in Lemma~\ref{lem:converg-phi}, we have
\begin{align} \label{limit:J-Eps-2}
    J_2^{\Eps_k} \to 
    \begin{cases}
      - \tilde D_1 |\xi|^{2} \hat\Psi_0 \,, & \alpha > 2 \,,  \\[2pt]
      - \tilde D_2 |\xi|^{\alpha} \hat\Psi_0 \,, & 1 < \alpha < 2 \,, \\[2pt]
      - \tilde D_3 |\xi|^{2} \hat\Psi_0 \,, & \alpha = 2
    \end{cases}
\qquad
   \text{weakly in $L^2(\R^n)$.}
\end{align}
The $\tilde D_j$'s correspond to the parameters in the three cases in Proposition~\ref{prop:Lambda-Eps} with $\kappa \equiv 1$. Since $\kappa$ is a constant, the special case in Proposition~\ref{prop:Lambda-Eps} applies and all the coefficients $\tilde D_j$'s are constants. 

Combining \eqref{converg:RHS}, \eqref{limit:EpsJ-1}, and \eqref{limit:J-Eps-2} we obtain the desired diffusion equation for $\Psi_0$. Moreover, since the solution to the diffusion equation in each case is unique in the space $L^2(\R^n)$, the limit holds along the full sequence $\{\hat\phi_\Eps\}$. 

\medskip

\noindent (c) Note that by $\int_0^\infty s p(s) \ds < \infty$, we have
\begin{align*}
    e^{-\int_0^s \Sigma_t(\tau) \dtau} \in L^1 \cap L^\infty(0, \infty) \,.
\end{align*}
Hence, for any $h \in L^2(\R^n)$, we have
\begin{align*}
    \int_0^\infty \int_{\R^n} \int_{\Ss^{n-1}}
       h(x) \Psi_\Eps(s, x, v) e^{-\int_0^s \Sigma_t(\tau) \dtau}
       \dv \dx \ds
\longrightarrow
  \beta_0 \int_0^\infty h(x) \Psi_0(x) \dx
\qquad
  \text{as $\Eps \to 0$}.
\end{align*}
where 
\begin{align} \label{def:beta-0}
    \beta_0 = \int_0^\infty e^{-\int_0^s \Sigma_t(\tau) \dtau} \dtau < \infty \,.
\end{align}
Therefore, we have
\begin{align*}
    \eta_\Eps = \int_0^\infty \int_{\Ss^{n-1}}\Psi_\Eps(s, x, v) \dv \ds
\to
    \beta_0 \Psi_0
\qquad
   \text{weakly in $L^2(\R^n)$.}
\end{align*}
By Part (a), the limiting equations for $\beta_0 \Psi_0$ are in the same format with the source term replaced by $\beta_0 \int_{\Ss^{n-1}} q(x, v) \dv$.
%
%
\end{proof}

\begin{rmk}
Note that in the case where $\alpha > 2$, there are two parts that contribute to the diffusion coefficient $D_1$ such that
\begin{align*}
 D_1 
   = \frac{1}{3} \nu_1 + \tilde D_1
   = \frac{1}{3} \nu_1 + \frac{1}{2} \vpran{\int_{\Ss^{n-1}} \vpran{v \cdot e_{\xi}}^2 \dv} D_0
   = \frac{1}{3} \nu_1 + \frac{1}{6}  D_0 \,,
\end{align*}
where $\nu_1$ and $D_0$ are defined in~\eqref{def:nu-1} and~\eqref{cond:p-1} respectively. This coefficient is consistent with the one in~\cite{Lar07} and captures anisotropic scattering. Interestingly, the anisotropy of the scattering vanishes from the limit equation in the heavy-tail case.
\end{rmk}

\section{Concluding Remarks and Future Work}
\label{sec:lorentz}
In a series of papers, Golse et al.\ (for a review cf.\ \cite{Gol12}), and independently Marklof \& Str\"ombergsson \cite{MarStr11} show that an equation similar to the non-classical transport equation can be derived from particle transport in a regular lattice. A test particle moves between obstacles that are placed on a regular lattice, and undergoes specular reflections. In the Boltzmann-Grad limit of shrinking obstacles, while simultaneously increasing their number so that the collision frequency is fixed, one obtains a kinetic equation that contains two seemingly unphysical memory variables, namely the distance to the next collision (similar to the variable $s$ in non-classical transport), as well as the impact factor for the next collision. This is the so-called periodic Lorentz gas equation. 

In 2D, an explicit path-length distribution can be computed. Translated into our notation, it reads
$$
p(s) = 
\begin{cases}  
\frac{24}{\pi^2} &\text{if } 0\leq s <\frac12,\\ 
\frac{24}{\pi^2}(\frac{1}{2s}+2(1-\frac{1}{2s})^2\ln(1-\frac{1}{2s})-\frac12 (1-\frac{1}{s})^2\ln(1-\frac{1}{s}))&\text{if } 0\leq s <\frac12.
\end{cases}
$$
As $s\to\infty$, this path length distribution behaves like
$$
p(s) \sim \frac{2}{\pi^2}\frac{1}{s^3} +\mathcal{O}(\frac{1}{s^4}).
$$
This means that the path length distribution of the periodic Lorentz gas corresponds exactly to the borderline case between classical and anomalous diffusion, as its second moment diverges logarithmically. We thus expect a classical diffusion equation with a non-classical coefficient in the asymptotic limit. For this simplified equation, this reproduces the result of Marklof \& T\'oth \cite{MarTot14} who proved a superdiffusive central limit theorem directly for the particle billards underlying the periodic Lorentz gas equation. They showed that the periodic Lorentz gas is superdiffusive, but only logarithmically.

There are several open topics related to non-classical transport. Among them are the formulation of correct boundary and interface conditions for heterogeneous media. In these media, it is also open how a fractional diffusion limit might look like. To study these questions, it will be necessary to generalize the classical kinetic theory technique to derive the diffusion limit, namely Hilbert expansion, to the fractional case. This has been done in \cite{AbdMelPue11}, although the decomposition that was used appears to be heavily inspired by the Fourier analysis. Furthermore, the asymptotic limit of the periodic Lorentz gas equation including impact factor should be studied, to see if the results of Marklof \& T\'oth \cite{MarTot14} can be retrieved by kinetic theory techniques.

\bibliographystyle{amsxport}
\bibliography{Radlit}

\begin{bibdiv}
\begin{biblist}

\bib{AbdMelPue11}{article}{
      author={Abdallah, N.~Ben},
      author={Mellet, A.},
      author={Puel, M.},
       title={Fractional diffusion limit for collisional kinetic equations: a
  {H}ilbert expansion approach},
        date={2011},
     journal={Kinet. Relat. Models},
      volume={4},
       pages={873\ndash 900},
}

\bib{AlbMar88}{article}{
      author={Albano, E.V.},
      author={Martin, H.O.},
       title={Temperature-programmed reactions with anomalous diffusion},
        date={1988},
     journal={J. Phys. Chem.},
      volume={92},
       pages={3594\ndash 3597},
}

\bib{BarSanSen84}{article}{
      author={Bardos, C.},
      author={Santos, R.},
      author={Sentis, R.},
       title={Diffusion approximation and computation of the critical size},
        date={1984},
     journal={Trans. Amer. Math. Soc.},
      volume={284},
       pages={617\ndash 649},
}

\bib{CarLynZas01}{article}{
      author={Carreras, B.A.},
      author={Lynch, V.E.},
      author={Zaslavsky, G.M.},
       title={Anomalous diffusion and exit time distribution of particle
  tracers in plasma turbulence model},
        date={2001},
     journal={Phys. Plasmas},
      volume={8},
       pages={5096},
}

\bib{Cercignani}{book}{
      author={Cercignani, C.},
       title={The {B}oltzmann equation and its applications},
   publisher={Springer-Verlag},
     address={New York},
        date={1988},
}

\bib{DavMar10}{article}{
      author={Davis, A.B.},
      author={Marshak, A.},
       title={Solar radiation transport in the cloudy atmosphere: A {3D}
  perspective on observations and climate impacts},
        date={2010},
     journal={Rep. Prog. Phys.},
      volume={73},
       pages={026801},
}

\bib{Eon14}{article}{
      author={d'Eon, E.},
       title={Rigorous asymptotic and moment-preserving diffusion
  approximations for generalized linear boltzmann transport in arbitrary
  dimension},
        date={2013},
     journal={Transp. Theory Stat. Phys.},
      volume={42},
       pages={237\ndash 297},
}

\bib{FraGou10}{article}{
      author={Frank, M.},
      author={Goudon, T.},
       title={On a generalized {B}oltzmann equation for non-classical particle
  transport},
        date={2010},
     journal={Kinet. Relat. Models},
      volume={3},
       pages={395\ndash 407},
}

\bib{Gol12}{incollection}{
      author={Golse, F.},
       title={Recent results on the periodic {L}orentz gas},
        date={2012},
   booktitle={Nonlinear partial differential equations},
   publisher={Springer-Verlag},
}

\bib{Lar07}{inproceedings}{
      author={Larsen, E.W.},
       title={A generalized {Boltzmann} equation for non-classical particle
  transport},
organization={American Nuclear Society},
        date={2007},
   booktitle={Joint international topical meeting on mathematics {\&}
  computation and supercomputing in nuclear applications},
        note={on CD-ROM},
}

\bib{LarVas11}{article}{
      author={Larsen, E.W.},
      author={Vasques, R.},
       title={A generalized linear boltzmann equation for non-classical
  particle transport},
        date={2011},
     journal={J. Quant. Spectrosc. Radiat. Transfer},
      volume={112},
       pages={619\ndash 631},
}

\bib{MarStr11}{article}{
      author={Marklof, J.},
      author={Str{\"o}mbergsson, A.},
       title={The {B}oltzmann-{G}rad limit of the periodic {L}orentz gas},
        date={2011},
     journal={Ann. of Math.},
      volume={174},
       pages={225\ndash 298},
}

\bib{MarTot14}{article}{
      author={Marklof, J.},
      author={Toth, B.},
       title={Superdiffusion in the periodic {L}orentz gas},
        date={2016},
     journal={Commun. Math. Phys.},
       pages={to appear},
}

\bib{MelMisMou11}{article}{
      author={Mellet, A.},
      author={Mischler, S.},
      author={Mouhot, C.},
       title={Fractional diffusion limit for collisional kinetic equations},
        date={2011},
     journal={Arch. Ration. Mech. Anal.},
      volume={199},
       pages={493\ndash 525},
}

\bib{MetKla00}{article}{
      author={Metzler, R.},
      author={Klafter, J.},
       title={The random walk's guide to anomalous diffusion: a fractional
  dynamics approach},
        date={2000},
     journal={Phys. Reports},
      volume={339},
       pages={1\ndash 77},
}

\bib{NMP2011}{article}{
      author={Naoufel, B.A.},
      author={Mellet, A.},
      author={Puel, M.},
       title={Anomalous diffusion limit for kinetic equations with degenerate
  collision frequency.},
        date={2011},
     journal={Math. Meth. Mod. Appl. Sci.},
      volume={21},
      number={11},
       pages={2249\ndash 2262},
}

\bib{Pfe99}{article}{
      author={Pfeilsticker, K.},
       title={First geometrical path lengths probability density function
  derivation of the skylight from spectroscopically highly resolving oxygen
  {A}-band observations. 2. derivation of the levy-index for the skylight
  transmitted by mid-latitude clouds},
        date={1999},
     journal={J. Geophys. Res.},
      volume={104},
       pages={4104\ndash 4116},
}

\bib{SagBroAmoDav12}{article}{
      author={Sagi, Y.},
      author={Brook, M.},
      author={Almog, I.},
      author={Davidson, N.},
       title={Observation of anomalous diffusion and fractional self-similarity
  in one dimension},
        date={2012},
     journal={Phys. Rev. Lett.},
      volume={108},
       pages={093002},
}

\bib{ScaGorMaiRab03}{article}{
      author={Scales, E.},
      author={Gorenflo, R.},
      author={Mainardi, F.},
      author={Raberto, M.},
       title={Revisiting the derivation of the fractional diffusion equation},
        date={2003},
     journal={Fractals},
      volume={11},
       pages={281\ndash 289},
}

\bib{SchHanDel11}{article}{
      author={Schumacher, E.},
      author={Hanert, E.},
      author={Deleersnijder, E.},
       title={Front dynamics in fractional-order epidemic models},
        date={2011},
     journal={J. Theo. Biol.},
      volume={279},
       pages={9\ndash 16},
}

\bib{VasLar09}{inproceedings}{
      author={Vasques, R.},
      author={Larsen, E.W.},
       title={Anisotropic diffusion in model 2-d pebble-bed reactor cores},
organization={American Nuclear Society},
        date={2009},
   booktitle={Joint international topical meeting on mathematics {\&}
  computation and supercomputing in nuclear applications},
        note={on CD-ROM},
}

\bib{VisAfaBulMurPriSta96}{article}{
      author={Viswanathan, G.M.},
      author={Afanasyev, V.},
      author={Buldyrev, S.V.},
      author={Murphy, E.J.},
      author={Prince, P.A.},
      author={Stanley, H.E.},
       title={Levy flight search patterns of wandering albatrosses},
        date={1996},
     journal={Nature},
      volume={381},
       pages={413\ndash 415},
}

\bib{VynBurRibWie12}{article}{
      author={Vynck, K.},
      author={Burresi, M.},
      author={Riboli, F.},
      author={Wiersma, D.},
       title={Photon management in two-dimensional disordered media},
        date={2012},
     journal={Nature Materials},
      volume={11},
       pages={1017\ndash 1022},
}

\end{biblist}
\end{bibdiv}

\end{document}